\documentclass[10pt, a4paper]{amsart}
\usepackage[dvips,final]{graphics}
\usepackage{array}
\usepackage{arydshln}
\usepackage[makeroom]{cancel}
 \usepackage[all]{xy}
 \usepackage{url}
\usepackage{multirow, blkarray}
\usepackage{booktabs}
\usepackage{textcomp}
\usepackage[final]{epsfig}
\usepackage{color}
\usepackage[T1]{fontenc}      
\usepackage[english,french]{babel}
\usepackage[utf8]{inputenc}
\usepackage{blindtext}

\usepackage{amsfonts,amscd,array, mathdots, epigraph}
\usepackage{amsmath}
\usepackage{amssymb}
\usepackage{amsthm}
\usepackage{mathrsfs}
\usepackage{stmaryrd}
\usepackage{diagbox}
\usepackage{enumitem}
\usepackage{ulem}

\usepackage{ulem}
\usepackage{tikz}
\usepackage{xcolor}
\usepackage{multicol}

\usepackage{fullpage}

\newtheorem{theorem}{Theorem}[section]

\newtheorem{proposition}[theorem]{Proposition}
\newtheorem{corollary}[theorem]{Corollary}

\newtheorem{lemma}[theorem]{Lemma}

\newtheorem*{remark}{Remark}

\numberwithin{equation}{section}

\newcommand{\bZ}{\mathbb{Z}}

\newcommand{\bN}{\mathbb{N}}
\newcommand{\bM}{\mathbb{N^{*}}}

\title{Some counting formulas for $\lambda$-quiddities over the rings $\bZ/2^{m}\bZ$}

\author{Flavien Mabilat}

\date{}


\email{flavien.mabilat@univ-reims.fr}

\begin{document}

\maketitle

\selectlanguage{french}
\begin{abstract}
Les $\lambda$-quiddités de taille $n$ sont des $n$-uplets d'éléments d'un ensemble fixé, solutions d'une équation matricielle apparaissant lors de l'étude des frises de Coxeter. Leur nombre et leurs propriétés sont intimement liés à la structure et au cardinal de l'ensemble choisi. L'objectif principal de ce texte est d'obtenir une formule explicite donnant le nombre de $\lambda$-quiddités de taille impaire, et un encadrement du nombre de $\lambda$-quiddités de taille paire, sur les anneaux $\mathbb{Z}/2^{m}\mathbb{Z}$ ($m \geq 2$). On donnera également des formules explicites concernant le nombre de $\lambda$-quiddités de taille $n$ sur $\mathbb{Z}/8\mathbb{Z}$.
\end{abstract}

\selectlanguage{english}
\begin{abstract}
The $\lambda$-quiddities of size $n$ are $n$-tuples of elements of a fixed set, solutions of a matrix equation appearing in the study of Coxeter's friezes. Their number and their properties are closely linked to the structure and the cardinality of the chosen set. The main objective of this text is to obtain an explicit formula giving the number of $\lambda$-quiddities of odd size, and a lower and upper bound for the number of $\lambda$-quiddities of even size, over the rings $\bZ/2^{m}\bZ$ ($m \geq 2$). We also give explicit formulas concerning the number of $\lambda$-quiddities of size $n$ over $\bZ/8\bZ$.
\end{abstract}

\ \\
\begin{flushright}
\textit{\og Sauf erreur, je ne me trompe jamais \fg} 
\\Alexandre Vialatte, \textit{Chroniques de La Montagne.}
\end{flushright}
\ \\

\section{Introduction}
\label{Intro}

Coxeter's friezes are mathematical objects which are closely linked to many topics (see for example \cite {Mo1}). They have been introduced at the beginning of the 1970s by the British mathematician H.\ S.\ M.\ Coxeter (see \cite{Cox}) and are defined as tables of numbers, belonging to a fixed set, having a finite number of lines of infinite length, arranged with an offset, and for which some arithmetic relations are verified. One of the main elements of the study of Coxeter's friezes is the resolution of the following equation over the chosen set :
\[M_{n}(a_1,\ldots,a_n):=\begin{pmatrix}
    a_{n} & -1 \\[4pt]
     1 & 0
    \end{pmatrix}
\begin{pmatrix}
    a_{n-1} & -1 \\[4pt]
     1 & 0
    \end{pmatrix}
    \cdots
    \begin{pmatrix}
    a_{1} & -1 \\[4pt]
     1 & 0
     \end{pmatrix}=-Id.\]
 
\noindent In particular, the intervention of the matrices $M_{n}(a_1,\ldots,a_n)$ is very interesting since they are involved in the study of many other mathematical objects, such as "negative" continued fractions, or the discrete Sturm-Liouville equations.
\\
\\ \indent Furthermore, the study of the previous equation naturally leads to consider the generalized equation below over a subset $R$ of a commutative and unitary ring $A$ :
\begin{equation}
\label{p}
\tag{$E_{R}$}
M_{n}(a_1,\ldots,a_n)=\pm Id.
\end{equation}
We will say that a solution $(a_1,\ldots,a_n)$ of \eqref{p} is a $\lambda$-quiddity of size $n$ over $R$ (if there is no ambiguity we will omit the set over which we are working) and our goal is to study these objects over different sets. There are several ways to achieve this objective. For example, we can try to find a recursive construction and a combinatorial description of the solutions. For instance, we have precise results about the solutions of $(E_{\mathbb{N}^{*}})$ (see \cite{O}). We can also define a notion of irreducible solutions and study them (see for example \cite{CH,M1,M2}). However, we can also, and this is what we will do here, look for general informations, such as the number of solutions of fixed size. In this direction, we already have formulas for $R=\mathbb{N}^{*}$ (see \cite{CO}) and for $R=\mathbb{F}_{q}$. We will recall in details the results obtained in this case. For $q$ the power of a prime number $p$, $B \in SL_{2}(\mathbb{F}_{q})$ and $n \in \mathbb{N}^{*}$, we define
$u_{n,q}^{+}:=|\{(a_{1},\ldots,a_{n})\in\mathbb{F}_{q}^{n},~M_{n}(a_{1},\ldots,a_{n})=Id\}|$ and $u_{n,q}^{-}:=|\{(a_{1},\ldots,a_{n})\in\mathbb{F}_{q}^{n},~M_{n}(a_{1},\ldots,a_{n})=-Id\}|$. Moreover, if $m \in \bM$ and $k \geq 2$, we write $[m]_{k}:=\frac{k^{m}-1 }{k-1}$ and $\binom{m}{2}_{k}:=\frac{(k^{m}-1)(k^{m-1}-1)}{(k-1)(k^{2}-1)}.$

\begin{theorem}[Morier-Genoud,~\cite{Mo2} Theorem 1]
\label{11}

Let $q$ be the power of a prime number $p$ and $n>4$.
\\
\\i) If $n$ is odd then $u_{n,q}^{-}=\left[\frac{n-1}{2}\right]_{q^{2}}$.
\\
\\ii) If $n$ is even then there exists $m \in \bM$ such that $n=2m$.
\begin{itemize}

\item If $p=2$, $u_{n,q}^{-}=(q-1)\binom{m}{2}_{q}+q^{m-1}$.
\item If $p>2$ and $m$ even we have: $u_{n,q}^{-}=(q-1)\binom{m}{2}_{q}$.
\item If $p>2$ and $m \geq 3$ odd we have: $u_{n,q}^{-}=(q-1)\binom{m}{2}_{q}+q ^{m-1}$.

\end{itemize}

\end{theorem}

\noindent Another proof of this result can also be found in \cite{S}.

\begin{theorem}[\cite{CM} Theorem 1.1]
\label{12}

Let $q$ be the power of a prime number $p>2$ and $n\in\mathbb{N}$, $n>4$.
\\
\\i) If $n$ is odd then we have $u_{n,q}^{+}=u_{n,q}^{-}=\left[\frac{n-1}{2}\right]_{q^{2}}$.
\\
\\
ii) If $n$ is even then there exists $m \in \bM$ such that $n=2m$.
\begin{itemize}
\item If $m$ is even we have: $u_{n,q}^{+}=(q-1)\binom{m}{2}_{q}+q^{m-1}$.
\item If $m \geq 3$ is odd we have: $u_{n,q}^{+}=(q-1)\binom{m}{2}_{q}$.
\end{itemize}

\end{theorem}

In this text, we will consider the case of the rings $\bZ/N\bZ$, that is to say we will be interested in the equations $(E_{\bZ/N\bZ})=(E_{N})$. Note that the resolution of $(E_{N})$ is linked to the the different writings of the elements of the congruence subgroup below :
\[\hat{\Gamma}(N):=\{C \in SL_{2}(\mathbb{Z}),~C= \pm Id~[N]\}.\]
Indeed, we know that all the matrices of $SL_{2}(\bZ)$ can be written in the form $M_{n}(a_1,\ldots,a_n)$, with $a_{i}$ a positive integer. Since this expression is not unique, we are naturally led to look for all the writings of this form for a given matrix, or a set of matrices. Note that we already have a lot of results concerning the solutions of $(E_{N})$ (see for example \cite{M1,M3}).
\\
\\ \indent Our objective in this article is to obtain the number of $\lambda$-quiddities of odd size, and an upper and lower bound of the number of $\lambda$-quiddities of even size, over the rings $\bZ/2^ {m}\bZ$. We will also give a complete formula for $\bZ/8\bZ$. For this, we write, for $n \geq 2$, $m \geq 2$ and $\epsilon \in \{-1,1\}$, $\Omega_{n}^{\epsilon}(m):=\{(a_{1},\ldots,a_{n}) \in (\mathbb{Z}/2^{m}\mathbb{Z})^{n},~M_{n}(a_{ 1},\ldots,a_{n})=\epsilon Id\}$, $w_{n,2^{m}}^{+}:=\left|\Omega_{n}^{1}(m) \right|$, $w_{n,2^{m}}^{-}:=\left|\Omega_{n}^{-1}(m)\right|$ and $w_{n,2^{m}}:=w_{n,2^{m}}^{+}+w_{n,2^{m}}^{-}$. We already have the following result :

\begin{theorem}[\cite{CM} Theorem 1.3]
\label{13}

Let $n\geq 3$.
\\
\\i) If $n$ is odd we have the following equality:
\[w_{n, 4}^{+}=w_{n, 4}^{-}=\frac{4^{n-2}-2^{n-3}}{3}.\]

\noindent ii) If $n$ is even then there exists $m \in \bM$ such that $n=2m$.
\begin{itemize}

\item If $m$ is even we have:
\[w_{n,4}^{+}=\frac{4^{n-2}+4 \times 2^{n-3}}{3}~~~and~~~w_{n,4 }^{-}=\frac{4^{n-2}- 2^{n-2}}{3}.\]
\item If $m$ is odd we have:
\[w_{n,4}^{+}=\frac{4^{n-2}- 2^{n-2}}{3}~~~and~~~w_{n,4}^{ -}=\frac{4^{n-2}+4 \times 2^{n-3}}{3}.\]

\end{itemize}

\end{theorem}

\noindent In this text, we will prove the two following results :

\begin{theorem}
\label{prin1}

Let $m \geq 2$ and $\epsilon \in \{-1,1\}$.
\\
\\i) Let $n \geq 2$. We have the following equality :
\[w_{2n+1,2^{m}}^{+}=w_{2n+1,2^{m}}^{-}=\frac{2^{2mn-2n-2m-1}(2^{2n+3}-8)}{3}.\]

\noindent ii) Let $n \geq 3$. We have the two following inequalities :
\begin{itemize}
\item $\left|\Omega_{2n}^{\epsilon}(m)\right| \geq \left|\Delta_{2n}^{\epsilon}(m)\right|+2^{m-1}\left|\Delta_{2n-1}^{\epsilon}(m)\right|+m2^{m-1}\left|\Delta_{2n-4}^{\epsilon}(m)\right|$;
\item $\left|\Omega_{2n}^{\epsilon}(m)\right| \leq \left|\Delta_{2n}^{\epsilon}(m)\right|+2^{m-1}\left|\Delta_{2n-1}^{\epsilon}(m)\right|+2^{2m-2}\left|\Delta_{2n-1}^{\epsilon}(m)\right|$.
\end{itemize}

\noindent with $\left|\Delta_{n}^{\epsilon}(m)\right|:=\frac{2^{mn-n-3m}(2^{n+1}+8 \times (-1)^{n+1})}{3}$.

\end{theorem}

\begin{theorem}
\label{prin2}

Let $n \geq 2$. We have the two following fomulas :
\[w_{2n+1,8}^{+}=w_{2n+1,8}^{-}=\frac{2^{6n-2n-7}(2^{2n+3}-8)}{3};\]
\[w_{2n,8}=28 \times 8^{n-2}+\frac{2^{4n-5}-2^{3n-3}+2^{6n-6}-2^{3n}}{3}.\]

\end{theorem}

\section{Proofs of the counting formulas}
\label{proof}

\subsection{Preliminary results}
\label{chap21}

The aim of this section is to provide some elements which will be useful in the proofs of our main theorems. In all this section, $U(m):=\{x \in \bZ/2^{m}\bZ,~x~{\rm invertible}\}$.

\begin{proposition}[\cite{CM} lemma 2.16 and proposition 2.18]
\label{21}

i) Let $A$ be a commutative and unitary ring. Let $n=2l \geq 4$ and $\lambda$ an invertible element of  $A$. Let $(a_{1},\ldots,a_{n}) \in A^{n}$. If $M_{n}(a_{1},\ldots,a_{n})=\epsilon Id$, with $\epsilon \in \{1,-1\}$, then $M_{n}(\lambda a_{1},\lambda^{-1} a_{2},\ldots,\lambda a_{2l-1},\lambda^{-1} a_{2l})=\epsilon Id$.
\\
\\ii) Let $A$ be a commutative and unitary ring. Let $n \in \bN^{*}$, $n$ odd. The application
\[\begin{array}{ccccc}  
\varphi_{n} & : & \{(a_{1},\ldots,a_{n}) \in A^{n},~M_{n}(a_{1},\ldots,a_{n})=Id\} & \longrightarrow & \{(a_{1},\ldots,a_{n}) \in A^{n},~M_{n}(a_{1},\ldots,a_{n})=-Id\} \\
 & & (a_{1},\ldots,a_{n}) & \longmapsto & (-a_{1},\ldots,-a_{n})  \\
\end{array}\] 
\noindent is a bijection.

\end{proposition}

\begin{lemma}
\label{form}
Let $A$ be a commutative and unitary ring and $(a,b,c,u,v) \in A^{5}$.
\\
\\i) $M_{3}(a,1,b)=M_{2}(a-1,b-1)$. 
\\
\\ii) $M_{3}(a,-1,b)=-M_{2}(a+1,b+1)$.
\\
\\iii) We suppose $uv-1$ invertible. $M_{4}(a,u,v,b)=M_{3}(a+(1-v)(uv-1)^{-1},uv-1,b+(1-u)(uv-1)^{-1})$.
\\
\\iv) (\cite{C}) We suppose $v$ invertible and $x=((vb-1)(uv-1)-1)v^{-1}$ invertible. 
\[M_{5}(a,u,v,b,c)=M_{3}(a-(vb-2)x^{-1},x,c-(uv-2)x^{-1}).\]

\end{lemma}

\begin{proof}

These formulas can be verified by direct computations. Note that i), ii) and iii) are given in the section 4 of \cite{CH}. iv) was an important formula obtained by M.\ Cuntz during the preparation of \cite{CM}.

\end{proof}

\begin{proposition}
\label{22}

Let $N=2^{m}$, $m \geq 2$, $B \in SL_{2}(\mathbb{Z}/N\mathbb{Z})$ and $n>4$. We define the set $\Delta_{n}^{B}(m):=\{(a_{1},\ldots,a_{n}) \in (\mathbb{Z}/2^{m}\mathbb{Z})^{n},~M_{n}(a_{1},\ldots,a_{n})=B~and~a_{2} \in U(m)\}$. We have the following equality :
\[\left|\Delta_{n}^{B}(m)\right|=2^{m-1}\left|\Delta_{n-1}^{B}(m)\right|+2^{2m-1}\left|\Delta_{n-2}^{B}(m)\right|.\]

\end{proposition}

\begin{proof}

Let $m \geq 2$, $B \in SL_{2}(\mathbb{Z}/2^{m}\mathbb{Z})$ and $n>4$. Let us begin by defining some elements :
\begin{itemize}
\item $\Omega_{n}^{B}(m):=\{(a_{1},\ldots,a_{n}) \in (\mathbb{Z}/2^{m}\mathbb{Z})^{n},~M_{n}(a_{1},\ldots,a_{n})=B\}$;
\item $\Delta_{n}^{B}(m):=\{(a_{1},\ldots,a_{n}) \in \Omega_{n}^{B}(m),~a_{2} \in U(m)\}$;
\item $\Lambda_{n}^{B}(m,x):=\{(a_{1},\ldots,a_{n}) \in \Omega_{n}^{B}(m),~a_{2}=x\}$;
\item $\begin{array}{ccccc} 
\psi : & U(m) \times U(m) \times \mathbb{Z}/2^{m}\mathbb{Z} & \longrightarrow & U(m) \\
  & (u,v,w) & \longmapsto & ((vw-1)(uv-1)-1)v^{-1}  \\
\end{array}$;
\item $T(m,x):=\{(u,v,w) \in U(m) \times U(m) \times \mathbb{Z}/2^{m}\mathbb{Z}, \psi(u,v,w)=x\}$, $x \in U(m)$.
\end{itemize}

\noindent We have the following equalities :
\begin{eqnarray*}
\Delta_{n}^{B}(m) &=& \bigsqcup_{u \in U(m)} \{(a_{1},\ldots,a_{n}) \in \Omega_{n}^{B}(m),~a_{2}=u\} \\
                  &=& \bigsqcup_{u \in U(m)} \underbrace{\{(a_{1},\ldots,a_{n}) \in \Omega_{n}^{B}(m),~a_{2}=u~{\rm and}~a_{3} \notin U(m)\}}_{X_{n}^{B}(u)} \\
									& & \bigsqcup_{u \in U(m)} \bigsqcup_{v \in U(m)} \underbrace{\{(a_{1},\ldots,a_{n}) \in \Omega_{n}^{B}(m),~a_{2}=u~{\rm and}~a_{3}=v\}}_{Y_{n}^{B}(u,v)}.
\end{eqnarray*}

\noindent Let $u=a+2^{m}\bZ \in U(m)$ ($a\in \bZ$). We consider the sets $X_{n}^{B}(u)$ and $Y_{n}^{B}(u,v)$ separately. 
\\
\\We begin by $X_{n}^{B}(u)$. Let $v=b+2^{m}\mathbb{Z}$ ($b \in \bZ$) be a non invertible element of $\mathbb{Z}/2^{m}\mathbb{Z}$. $b$ is an even integer. Hence, $ab$ is an even integer and $ab-1$ is an odd integer. Thus, $uv-1 \in U(m)$. By lemma \ref{form} iii), we can define the two following applications :

\[\begin{array}{ccccc} 
f_{n,u} : & X_{n}^{B}(u) & \longrightarrow & \Delta_{n-1}^{B}(m) \\
  & (a_{1},u,a_{3},\ldots,a_{n}) & \longmapsto & (a_{1}+(1-a_{3})(ua_{3}-1)^{-1},ua_{3}-1,a_{4}+(1-u)(ua_{3}-1)^{-1},a_{5},\ldots,a_{n})  \\
\end{array}\]
\noindent and
\[\begin{array}{ccccc} 
g_{n,u} : & \Delta_{n-1}^{B}(m) & \longrightarrow & X_{n}^{B}(u) \\
 & (a_{1},\ldots,a_{n-1}) & \longmapsto & (a_{1}+(u^{-1}(a_{2}+1)-1)a_{2}^{-1},u,u^{-1}(a_{2}+1),a_{3}+(u-1)a_{2}^{-1},a_{4},\ldots,a_{n-1})  \\
\end{array}.\]

\noindent $f_{n,u}$ and $g_{n,u}$ are reciprocal bijections. Hence, $\left|X_{n}^{B}(u)\right|=\left|\Delta_{n-1}^{B}(m)\right|$.
\\
\\Let $v=b+2^{m}\bZ$ ($b \in \bZ$) be an invertible elements of $\mathbb{Z}/2^{m}\mathbb{Z}$. Now, we consider $Y_{n}^{B}(u,v)$. We have :
\[Y_{n}^{B}(u,v)=\bigsqcup_{w \in \mathbb{Z}/2^{m}\mathbb{Z}} \underbrace{\{(a_{1},\ldots,a_{n}) \in \Omega_{n}^{B}(m), a_{2}=u,~a_{3}=v~{\rm and}~a_{4}=w\}}_{Z_{n}^{B}(u,v,w)}.\]
Hence, we will consider the sets $Z_{n}^{B}(u,v,w)$. Let $w$ be an element of $\mathbb{Z}/2^{m}\mathbb{Z}$. $a$ and $b$ are odd integers. So, $ab$ is an odd integer and $ab-1$ is an even integer. Hence, $uv-1$ is not invertible. Let $x=\psi(u,v,w)=((vw-1)(uv-1)-1)v^{-1}$. Since $uv-1$ is not invertible, $x$ is invertible. By lemma \ref{form} iv), we can define the two following applications :
\[\begin{array}{ccccc} 
h_{n,u,v,w} : & Z_{n}^{B}(u,v,w) & \longrightarrow & \Lambda_{n-2}^{B}(m,x) \\
  & (a_{1},u,v,w,a_{5},\ldots,a_{n}) & \longmapsto & (a_{1}-(vw-2)x^{-1},x,a_{5}-(uv-2)x^{-1},a_{6},\ldots,a_{n})  \\
\end{array}\]
\[{\rm and}~~\begin{array}{ccccc} 
k_{n,u,v,w} : & \Lambda_{n-2}^{B}(m,x) & \longrightarrow & Z_{n}^{B}(u,v,w) \\
 & (a_{1},\ldots,a_{n-2}) & \longmapsto & (a_{1}+(vw-2)x^{-1},u,v,w,a_{3}+(uv-2)x^{-1},a_{4},\ldots,a_{n-2})  \\
\end{array}.\]

\noindent $h_{n,u,v,w}$ and $k_{n,u,v,w}$ are reciprocal bijections. Hence, $\left|Z_{n}^{B}(u,v,w)\right|=\left|\Lambda_{n-2}^{B}(m,x)\right|$.
\\
\\Now, we give some elements about the sets $T(m,x)$. First, we have \[U(m) \times U(m) \times (\mathbb{Z}/2^{m}\mathbb{Z})=\bigsqcup_{x \in U(m)} \{(u,v,w) \in U(m) \times U(m) \times \mathbb{Z}/2^{m}\mathbb{Z}, \psi(u,v,w)=x\}.\]
\noindent Let $x \in U(m)$. We define two applications :
\[\begin{array}{ccccc} 
\alpha_{x} : & T(m,1) & \longrightarrow & T(m,x) \\
  & (u,v,w) & \longmapsto & (xu,vx^{-1},wx)  \\
\end{array}~~\begin{array}{ccccc}
\beta_{x} : & T(m,x) & \longrightarrow & T(m,1) \\
 & (u,v,w) & \longmapsto & (ux^{-1},vx,wx^{-1})  \\
\end{array}.\]

\noindent $\alpha_{x}$ and $\beta_{x}$ are reciprocal bijections. Hence, $\left|T(m,x)\right|=\left|T(m,1)\right|$. Moreover, 
\[2^{3m-2}=\left|U(m) \times U(m) \times (\mathbb{Z}/2^{m}\mathbb{Z})\right|=\sum_{x \in U(m)} \left|T(m,x)\right|=2^{m-1}\left|T(m,1)\right|.\] 
\noindent So, $\left|T(m,x)\right|=\left|T(m,1)\right|=2^{2m-1}$.
\\
\\If we collect all the preceedings elements, we have :
\begin{eqnarray*} 
\left|\Delta_{n}^{B}(m)\right| &=& \sum_{u \in U(m)} \left|X_{n}^{B}(u)\right|+\sum_{u, v \in U(m)} \left|Y_{n}^{B}(u,v)\right| \\
                               &=& \sum_{u \in U(m)} \left|\Delta_{n-1}^{B}(m)\right|+\sum_{u, v \in U(m)} \left(\sum_{w \in \mathbb{Z}/2^{m}\mathbb{Z}} \left|Z_{n}^{B}(u,v,w)\right|\right) \\
												       &=& \left|U(m)\right|\left|\Delta_{n-1}^{B}(m)\right|+\sum_{x \in U(m)} \left(\sum_{(u,v,w) \in T(m,x)} \left|Z_{n}^{B}(u,v,w)\right|\right) \\
															 &=& 2^{m-1} \left|\Delta_{n-1}^{B}(m)\right|+\sum_{x \in U(m)} \left(\sum_{(u,v,w) \in T(m,x)} \left|\Lambda_{n-2}^{B}(m,x)\right|\right). \\
															 &=& 2^{m-1} \left|\Delta_{n-1}^{B}(m)\right|+\sum_{x \in U(m)} \left|T(m,x)\right| \left|\Lambda_{n-2}^{B}(m,x)\right|. \\
															 &=& 2^{m-1} \left|\Delta_{n-1}^{B}(m)\right|+\left|T(m,1)\right|\sum_{x \in U(m)} \left|\Lambda_{n-2}^{B}(m,x)\right|. \\
														   &=& 2^{m-1} \left|\Delta_{n-1}^{B}(m)\right|+\left|T(m,1)\right|\left|\Delta_{n-2}^{B}(m)\right| \\
															 &=& 2^{m-1} \left|\Delta_{n-1}^{B}(m)\right|+2^{2m-1}\left|\Delta_{n-2}^{B}(m)\right|. \\
\end{eqnarray*}

\end{proof}

\begin{remark}
{\rm Let $x \neq y$ be two invertible elements of $\mathbb{Z}/2^{m}\mathbb{Z}$. In general, $\left|\Lambda_{n}^{B}(m,x)\right| \neq \left|\Lambda_{n}^{B}(m,y)\right|$. For instance, computationnaly, we find the following values: $\left|\Lambda_{5}^{Id}(3,1+8\mathbb{Z})\right|=20$ and $\left|\Lambda_{5}^{Id}(3,3+8\mathbb{Z})\right|=8$.
}
\end{remark}

\begin{proposition}
\label{23}

Let $N=2^{m}$, $m \geq 2$, $B \in SL_{2}(\mathbb{Z}/N\mathbb{Z})$ and $n>4$. We define the set $\Delta_{n}^{B}(m):=\{(a_{1},\ldots,a_{n}) \in (\mathbb{Z}/N\mathbb{Z})^{n},~M_{n}(a_{1},\ldots,a_{n})=B~and~a_{2} \in U(m)\}$. 
\[\left|\Delta_{n}^{B}(m)\right|=\frac{2^{mn-n-4m+1}(2^{n}+(-1)^{n}\times 8)}{3}\left|\Delta_{4}^{B}(m)\right|+\frac{2^{mn-n-3m}(2^{n}+(-1)^{n+1}\times 16)}{3}\left|\Delta_{3}^{B}(m)\right|.\]

\end{proposition}

\begin{proof}

We set $A=\begin{pmatrix}
   2^{m-1} & 2^{2m-1} \\[4pt]
    1    & 0 
   \end{pmatrix}$ and $P=\begin{pmatrix}
   1 & 1 \\[4pt]
   \frac{1}{2^{m}} & \frac{-1}{2^{m-1}} 
   \end{pmatrix}$. $P^{-1}=\frac{-2^{m}}{3}\begin{pmatrix}
   \frac{-1}{2^{m-1}} & -1 \\[4pt]
   \frac{-1}{2^{m}} & 1
   \end{pmatrix}$. By the previous proposition, we have te following equality :
	\[\begin{pmatrix}
   \left|\Delta_{n}^{B}(m)\right|  \\[4pt]
    \left|\Delta_{n-1}^{B}(m)\right|  
   \end{pmatrix}=\begin{pmatrix}
   2^{m-1} & 2^{2m-1} \\[4pt]
    1    & 0 
   \end{pmatrix}\begin{pmatrix}
   \left|\Delta_{n-1}^{B}(m)\right|  \\[4pt]
    \left|\Delta_{n-2}^{B}(m)\right|  
   \end{pmatrix}=A^{n-4}\begin{pmatrix}
   \left|\Delta_{4}^{B}(m)\right|  \\[4pt]
    \left|\Delta_{3}^{B}(m)\right|  
   \end{pmatrix}.\]
	
\noindent Moreover, $A=P\begin{pmatrix}
   2^{m} & 0 \\[4pt]
   0 & -2^{m-1} 
   \end{pmatrix}P^{-1}$. Hence, 
	\[\left|\Delta_{n}^{B}(m)\right|=\frac{2^{mn-n-4m+1}(2^{n}+(-1)^{n}\times 8)}{3}\left|\Delta_{4}^{B}(m)\right|+\frac{2^{mn-n-3m}(2^{n}+(-1)^{n+1}\times 16)}{3}\left|\Delta_{3}^{B}(m)\right|.\]
\end{proof}

Let $S:=\begin{pmatrix}
     0 & -1 \\[4pt]
     1 & 0
    \end{pmatrix}$ and $T:=\begin{pmatrix}
     1 & 1 \\[4pt]
     0 & 1
    \end{pmatrix}$. $SL_{2}(\mathbb{Z})$ is generated by $S$ and $T$ and, for all $(a_{1},\ldots,a_{n}) \in \bZ^{n}$, $M_{n}(a_{1},\ldots,a_{n})=T^{a_{n}}S\ldots T^{a_{1}}S$.

\begin{corollary}
\label{24}
Let $m \geq 2$ and $n>4$. We have the followig formulas :
\begin{itemize}
\item $\left|\Delta_{n}^{Id}(m)\right|=\left|\Delta_{n}^{-Id}(m)\right|=\frac{2^{mn-n-3m}(2^{n+1}+8 \times (-1)^{n+1})}{3}$;
\item $\left|\Delta_{n}^{S}(m)\right|=\left|\Delta_{n}^{-S}(m)\right|=\frac{2^{mn-n-3m+1}(2^{n}+(-1)^{n}\times 8)}{3}$;
\item $\left|\Delta_{n}^{T}(m)\right|=\left|\Delta_{n}^{-T}(m)\right|=\left|\Delta_{n}^{Id}(m)\right|=\frac{2^{mn-n-3m}(2^{n+1}+8 \times (-1)^{n+1})}{3}$.
\end{itemize}

\end{corollary}

\begin{proof}

We apply the formula given in the previous proposition with the values :
\begin{itemize}
\item $\left|\Delta_{4}^{Id}(m)\right|=\left|\Delta_{4}^{-Id}(m)\right|=2^{m-1}$, $\left|\Delta_{3}^{Id}(m)\right|=\left|\Delta_{3}^{-Id}(m)\right|=1$;
\item $\left|\Delta_{4}^{S}(m)\right|=\left|\Delta_{4}^{-S}(m)\right|=2^{m}$, $\left|\Delta_{3}^{S}(m)\right|=\left|\Delta_{3}^{-S}(m)\right|=0$; 
\item $\left|\Delta_{4}^{T}(m)\right|=\left|\Delta_{4}^{-T}(m)\right|=2^{m-1}$, $\left|\Delta_{3}^{T}(m)\right|=\left|\Delta_{3}^{-T}(m)\right|=1$.
\end{itemize}
\end{proof}

\noindent For instance, we have the following values :

\hfill\break

\begin{center}
\begin{tabular}{|c|c|c|c|c|c|c|c|c|}
\hline
  $n$     & 3 & 4 & 5 & 6 & 7 & 8 & 9 & 10  \rule[-7pt]{0pt}{18pt} \\
  \hline
  $\left|\Delta_{n}^{Id}(3)\right|$ & 1 & 4 & 48 & 320 & 2816 & 21~504 & 176~128 & 1~392~640    \rule[-7pt]{0pt}{18pt} \\
	\hline
  $\left|\Delta_{n}^{S}(3)\right|$ & 0 & 8 & 32 & 384 & 2560 & 22~528 & 172~032 & 1~409~024    \rule[-7pt]{0pt}{18pt} \\
	\hline
	
\end{tabular}
\end{center}

\hfill\break

\begin{proposition}
\label{25}

Let $N=2^{m}$, $m \geq 2$, $B \in SL_{2}(\mathbb{Z}/N\mathbb{Z})$ and $n \geq 3$. We define the sets $\Omega_{n}^{B}(m):=\{(a_{1},\ldots,a_{n}) \in (\mathbb{Z}/2^{m}\mathbb{Z})^{n},~M_{n}(a_{1},\ldots,a_{n})=B\}$ and $\Lambda_{n}^{B}(m,-1):=\{(a_{1},\ldots,a_{n}) \in \Omega_{n}^{B}(m),~a_{2}=-1\}$.
\[\left|\Lambda_{n}^{B}(m,-1)\right|=\left|\Omega_{n-1}^{-B}(m)\right|.\]

\end{proposition}

\begin{proof}

By lemma \ref{form} ii), $(a_{1},\ldots,a_{n}) \in \Lambda_{n}^{B}(m,-1) \mapsto (a_{1}+1,a_{3}+1,a_{4},\ldots,a_{n}) \in \Omega_{n-1}^{-B}(m)$ is a bijection.

\end{proof}

\begin{lemma}
\label{26}

$m \geq 2$. $\left|\{(x,y) \in ((\mathbb{Z}/2^{m}\mathbb{Z})-U(m))^{2},~xy=0\}\right|=m2^{m-1}$.

\end{lemma}

\begin{proof}
We have the following equalities :
\begin{eqnarray*}
R &=& \{(x,y) \in ((\mathbb{Z}/2^{m}\mathbb{Z})-U(m))^{2},~xy=0\} \\
  &=& \{(0,y),~y \in (\mathbb{Z}/2^{m}\mathbb{Z})-U(m)\} \sqcup \{(x,0),~x \in (\mathbb{Z}/2^{m}\mathbb{Z})-U(m)~{\rm and}~x \neq 0\} \\
	& & \sqcup \{(2^{k}a+2^{m}\mathbb{Z},2^{l}b+2^{m}\mathbb{Z}),~1 \leq k \leq m-1,~m-k \leq l \leq m-1,~a,b~{\rm odd},~0< 2^{k}a, 2^{l}b <2^{m} \}.
\end{eqnarray*}
\noindent Hence, \begin{eqnarray*}
\left|R\right| &=& 2^{m-1}+2^{m-1}-1+\sum_{k=1}^{m-1} 2^{m-k-1} \sum_{l=m-k}^{m-1} 2^{m-l-1} \\
               &=& 2^{m}-1+2^{2m-2}\sum_{k=1}^{m-1} \frac{1}{2^{k}} \frac{2}{2^{m-k}}(1-2^{-k}) \\
							 &=& 2^{m}-1+2^{m-1}(m-1-(1-2^{-m+1})) \\
							 &=& m2^{m-1}.
\end{eqnarray*}

\end{proof}

\begin{proposition}
\label{27}

$m \geq 2$. $w_{4,2^{m}}^{+}=(m+2)2^{m-1}$, $w_{4,2^{m}}^{-}=2^{m}$.

\end{proposition}

\begin{proof}

i) $\Omega_{4}^{1}(m)=\{(-y,x,y,-x) \in (\mathbb{Z}/2^{m}\mathbb{Z})^{4},~xy=0\}$. Hence, by the previous lemma, we have the following equality :
\begin{eqnarray*}
\left|\Omega_{4}^{1}(m)\right| &=& \left|\{(x,y) \in ((\mathbb{Z}/2^{m}\mathbb{Z})-U(m))^{2},~xy=0\}\right|+\left|\{(x,0),~x \in U(m)\}\right|+\left|\{(0,y),~y \in U(m)\}\right| \\
                               &=& (m+2)2^{m-1}.
\end{eqnarray*}															

\noindent ii) $\Omega_{4}^{-1}(m)=\{(y,x,y,x) \in (\mathbb{Z}/2^{m}\mathbb{Z})^{4},~xy=2+2^{m}\mathbb{Z}\}$. Let $x=2a+2^{m}\mathbb{Z}$ and $y=2b+2^{m}\mathbb{Z}$, $(a,b) \in \mathbb{Z}^{2}$. $xy=4ab+2^{m}\mathbb{Z} \neq 2+2^{m}\mathbb{Z}$ since $4ab-2$ is not a multiple of $2^{m}$ ($m \geq 2$). So, $xy=2$ implies $x \in U(m)$ or $y \in U(m)$.
\\
\\Hence, we have the following equality :
\[\left|\Omega_{4}^{-1}(m)\right|=\left|\{(-2x^{-1},x,2x^{-1},-x),~x \in U(m)\}\right|+\left|\{(-y,2y^{-1},y,-2y^{-1}),~y \in U(m)\}\right|=2^{m}.\]

\end{proof}

\subsection{Proof of theorem \ref{prin1}}
\label{chap22}

Let $N=2^{m}$, $m \geq 2$, $n \geq 2$ and $\epsilon=\pm 1$. We define the sets :
\begin{itemize}
\item $\Omega_{n}^{\epsilon}(m):=\{(a_{1},\ldots,a_{n}) \in (\mathbb{Z}/2^{m}\mathbb{Z})^{n},~M_{n}(a_{1},\ldots,a_{n})=\epsilon Id\}$;
\item $\Delta_{n}^{\epsilon}(m):=\{(a_{1},\ldots,a_{n}) \in (\mathbb{Z}/2^{m}\mathbb{Z})^{n},~M_{n}(a_{1},\ldots,a_{n})=\epsilon Id~{\rm and}~a_{2} \in U(m)\}$;
\item $u \in U(m)$, $\Lambda_{n}^{\epsilon}(m,u)=\{(a_{1},\ldots,a_{n}) \in \Omega_{n}^{\epsilon}(m),~a_{2}=u\}$.
\end{itemize}

\noindent i) By proposition \ref{21} i) and lemma \ref{form} i), we can define the two following applications :
 \[\begin{array}{ccccc} 
\vartheta_{n,u} : & \Lambda_{2n}^{\epsilon}(m,u) & \longrightarrow & \Omega_{2n-1}^{\epsilon}(m) \\
  & (a_{1},u,a_{3},\ldots,a_{2n}) & \longmapsto & (a_{1}u-1,a_{3}u-1,a_{4}u^{-1},a_{5}u,\ldots,a_{2n}u^{-1})  \\
\end{array}\]
\[{\rm and}~~\begin{array}{ccccc} 
\theta_{n,u}  : & \Omega_{2n-1}^{\epsilon}(m) & \longrightarrow & \Lambda_{2n}^{\epsilon}(m,u) \\
 & (a_{1},\ldots,a_{2n-1}) & \longmapsto & ((a_{1}+1)u^{-1},u,(a_{2}+1)u^{-1},a_{3}u,\ldots,a_{2n-1}u)  \\
\end{array}.\]

\noindent $\vartheta_{n,u}$ and $\theta_{n,u}$ are reciprocal bijections. Hence, $\left|\Omega_{2n-1}^{\epsilon}(m)\right|=\left|\Lambda_{2n}^{\epsilon}(m,u)\right|$. So, $\left|\Lambda_{2n}^{\epsilon}(m,u)\right|=\left|\Lambda_{2n}^{\epsilon}(m,v)\right|$, for all $u,v \in U(m)$. 
\\
\\Besides, $\left|\Delta_{2n}^{\epsilon}(m)\right|=\sum_{u \in U(m)} \left|\Lambda_{2n}^{\epsilon}(m,u)\right|=2^{m-1}\left|\Omega_{2n-1}(m)\right|$. By corollary \ref{24}, 
\[w_{2n-1,2^{m}}^{+}=\left|\Omega_{2n-1}^{\epsilon}(m)\right|=\frac{2^{2mn-2n-3m}(2^{2n+1}+8 \times(-1)^{2n+1})}{3\times 2^{m-1}}=\frac{2^{2mn-2n-4m+1}(2^{2n+1}-8)}{3}.\]
\noindent So, by proposition \ref{21} ii), we have :
\[w_{2n+1,2^{m}}^{+}=w_{2n+1,2^{m}}^{-}=\frac{2^{2mn-2n-2m-1}(2^{2n+3}-8)}{3}.\]

\noindent ii) Let $n \geq 3$. We have the following equality :
\begin{eqnarray*}
\Omega_{2n}^{\epsilon}(m) &=& \Delta_{2n}^{\epsilon}(m) \bigsqcup_{x \notin U(m)}\bigsqcup_{y \in \mathbb{Z}/2^{m}\mathbb{Z}} \underbrace{\{(a_{1},\ldots,a_{2n}) \in \Omega_{2n}^{\epsilon}(m),~a_{2}=x,~a_{3}=y\}}_{G_{2n}^{\epsilon}(x,y)} \\
                          &=& \Delta_{2n}^{\epsilon}(m) \bigsqcup_{x \notin U(m)}\bigsqcup_{y \in U(m)} G_{2n}^{\epsilon}(x,y) \bigsqcup_{x \notin U(m)}\bigsqcup_{y \notin U(m)} G_{2n}^{\epsilon}(x,y). \\
\end{eqnarray*}

\noindent Let $x=a+2^{m}\mathbb{Z}$ ($a \in \mathbb{Z}$) be a non invertible element of $\mathbb{Z}/2^{m}\mathbb{Z}$. Let $y=b+2^{m}\mathbb{Z}$ ($b \in \mathbb{Z}$) . $a$ is an even integer. Hence, $ab$ is an even integer and $ab-1$ is an odd integer. Thus, $xy-1 \in U(m)$. By lemma \ref{form} iii), we can define the two following applications :

\[\begin{array}{ccccc} 
\sigma_{n,x,y} : & G_{2n}^{\epsilon}(x,y) & \longrightarrow & \Lambda_{2n-1}^{\epsilon}(m,xy-1) \\
  & (a_{1},x,y,a_{4},\ldots,a_{2n}) & \longmapsto & (a_{1}+(1-y)(xy-1)^{-1},xy-1,a_{4}+(1-x)(xy-1)^{-1},a_{5},\ldots,a_{2n})  \\
\end{array}\]

\[{\rm and}~~\begin{array}{ccccc} 
\tau_{n,x,y} : & \Lambda_{2n-1}^{\epsilon}(m,xy-1) & \longrightarrow & G_{2n}^{\epsilon}(x,y) \\
 & (a_{1},xy-1,a_{3},\ldots,a_{2n-1}) & \longmapsto & (a_{1}-(1-y)(xy-1)^{-1},x,y,a_{3}-(1-x)(xy-1)^{-1}, \\
               &                     &             & a_{4},\ldots,a_{2n-1})  \\
\end{array}~~~~~~.\]

\noindent $\sigma_{n,x,y}$ and $\tau_{n,x,y}$ are reciprocal bijections. Hence, $\left|G_{2n}^{\epsilon}(x,y)\right|=\left|\Lambda_{2n-1}^{\epsilon}(m,xy-1)\right|$.
\\
\\Let $y \in U(m)$, $x \in (\mathbb{Z}/2^{m}\mathbb{Z}-U(m)) \mapsto xy-1 \in U(m)$ is a bijection. Indeed, let $z \in U(m)$. The equation $xy-1=z$ has exactly one solution in $(\mathbb{Z}/2^{m}\mathbb{Z})-U(m)$ : $x=(z+1)y^{-1}$. Hence, 
\begin{eqnarray*}
\left|\bigsqcup_{x \notin U(m)}\bigsqcup_{y \in U(m)} G_{2n}^{\epsilon}(x,y)\right| &=& \sum_{y \in U(m)} \sum_{x \notin U(m)} \left|G_{2n}^{\epsilon}(x,y)\right|\\
                                                                                   &=& \sum_{y \in U(m)} \sum_{x \notin U(m)} \left|\Lambda_{2n-1}^{\epsilon}(m,xy-1)\right| \\
																																									 &=& \sum_{y \in U(m)} \sum_{z \in U(m)} \left|\Lambda_{2n-1}^{\epsilon}(m,z)\right| \\
																																									 &=& \sum_{y \in U(m)} \left|\Delta_{2n-1}^{\epsilon}(m)\right|\\
																																									 &=& 2^{m-1}\left|\Delta_{2n-1}^{\epsilon}(m)\right|.
\end{eqnarray*}

\noindent Moreover, $\left|\bigsqcup_{x,y \notin U(m)} G_{2n}^{\epsilon}(x,y)\right|=\sum_{x,y \notin U(m)} \left|G_{2n}^{\epsilon}(x,y)\right|=\sum_{x,y \notin U(m)} \left|\Lambda_{2n-1}^{\epsilon}(m,xy-1)\right|$. 
\\
\\So, $\left|\bigsqcup_{x,y \notin U(m)} G_{2n}^{\epsilon}(x,y)\right| \leq \sum_{x,y \notin U(m)} \left|\Delta_{2n-1}^{\epsilon}(m)\right|=2^{2m-2}\left|\Delta_{2n-1}^{\epsilon}(m)\right|$. 
\\
\\Besides, $\left|\bigsqcup_{x,y \notin U(m)} G_{2n}^{\epsilon}(x,y)\right| \geq \left|\bigsqcup_{x,y \notin U(m), xy=0} G_{2n}^{\epsilon}(x,y)\right| \geq m2^{m-1}\left|\Delta_{2n-4}^{\epsilon}(m)\right|$. Indeed, let $x,y$ be two non invertible elements of $\mathbb{Z}/2^{m}\mathbb{Z}$ verifying $xy=0$. $(-y,x,y,-x) \in \Omega_{4}^{1}(m)$. For all elements $(a_{1},\ldots,a_{2n-4}) \in \Delta_{2n-4}^{\epsilon}(m)$, $(-y,x,y,-x,a_{1},\ldots,a_{2n-4}) \in G_{2n}^{\epsilon}(x,y)$. So, $\left|G_{2n}^{\epsilon}(x,y)\right| \geq \left|\Delta_{2n-4}^{\epsilon}(m)\right|$. By combining this with the result of lemma \ref{26}, we have the desired inequality.
\\
\\Hence, $\left|\Delta_{2n}^{\epsilon}(m)\right|+2^{m-1}\left|\Delta_{2n-1}^{\epsilon}(m)\right|+m2^{m-1}\left|\Delta_{2n-4}^{\epsilon}(m)\right| \leq \left|\Omega_{2n}^{\epsilon}(m)\right| \leq \left|\Delta_{2n}^{\epsilon}(m)\right|+2^{m-1}\left|\Delta_{2n-1}^{\epsilon}(m)\right|+2^{2m-2}\left|\Delta_{2n-1}^{\epsilon}(m)\right|$. If we associate this inequality with the formula given in the corollary \ref{24}, we have the result given in the theorem.

\qed

\subsection{The case of $\mathbb{Z}/8\mathbb{Z}$}
\label{chap23}

We use the notations introduced in the previous section.
\\
\\We already have the formula for $w_{2n+1,8}^{+}=w_{2n+1,8}^{-}$. Let $n \geq 3$ and $\epsilon \in \{-1,1\}$ We will focus on $\left|\Omega_{2n}^{\epsilon}(3)\right|+\left|\Omega_{2n}^{-\epsilon}(3)\right|$. The proof of theorem \ref{prin1} gives us the following formula :
\[\left|\Omega_{2n}^{\epsilon}(3)\right|=\left|\Delta_{2n}^{\epsilon}(3)\right|+4\left|\Delta_{2n-1}^{\epsilon}(3)\right|+\sum_{x,y \notin U(3)} \left|\Lambda_{2n-1}^{\epsilon}(3,xy-1)\right|.\]

\noindent So, to have a complete formula, we have to study the value of $\sum_{x,y \notin U(3)} \left|\Lambda_{2n-1}^{\epsilon}(3,xy-1)\right|$. To do this, we will use the different possible values of $xy-1$ ($x,y \notin U(3)$) given in the following table :

\hfill\break

\begin{center}
\begin{tabular}{|c|c|c|c|c|}
\hline
  \multicolumn{1}{|c|}{\diagbox{$x$}{$y$}}     & $0+8\mathbb{Z}$ & $2+8\mathbb{Z}$ & $4+8\mathbb{Z}$ & $6+8\mathbb{Z}$   \rule[-7pt]{0pt}{18pt} \\
  \hline
  $0+8\mathbb{Z}$ & $-1+8\mathbb{Z}$ & $-1+8\mathbb{Z}$ & $-1+8\mathbb{Z}$ & $-1+8\mathbb{Z}$   \rule[-7pt]{0pt}{18pt} \\
	\hline
  $2+8\mathbb{Z}$ & $-1+8\mathbb{Z}$ & $3+8\mathbb{Z}$ & $-1+8\mathbb{Z}$ & $3+8\mathbb{Z}$    \rule[-7pt]{0pt}{18pt} \\
	\hline
  $4+8\mathbb{Z}$  & $-1+8\mathbb{Z}$ & $-1+8\mathbb{Z}$ & $-1+8\mathbb{Z}$ & $-1+8\mathbb{Z}$    \rule[-7pt]{0pt}{18pt} \\
	\hline
  $6+8\mathbb{Z}$  & $-1+8\mathbb{Z}$ & $3+8\mathbb{Z}$ & $-1+8\mathbb{Z}$ & $3+8\mathbb{Z}$  \rule[-7pt]{0pt}{18pt} \\
  \hline
	
\end{tabular}
\end{center}

\hfill\break

\noindent So, $\sum_{x,y \notin U(3)} \left|\Lambda_{2n-1}^{\epsilon}(3,xy-1)\right|=12\left| \Lambda_{2n-1}^{\epsilon}(3,-1+8\mathbb{Z})\right|+4\left| \Lambda_{2n-1}^{\epsilon}(3,3+8\mathbb{Z})\right|$. Hence,
\begin{eqnarray*}
w_{2n,8} &=& \left|\Omega_{2n}^{\epsilon}(3)\right|+\left|\Omega_{2n}^{-\epsilon}(3)\right| \\
         &=& \left|\Delta_{2n}^{\epsilon}(3)\right|+4\left|\Delta_{2n-1}^{\epsilon}(3)\right|+12\left| \Lambda_{2n-1}^{\epsilon}(3,-1+8\mathbb{Z})\right|+4\left| \Lambda_{2n-1}^{\epsilon}(3,3+8\mathbb{Z})\right| \\
				 &+& \left|\Delta_{2n}^{-\epsilon}(3)\right|+4\left|\Delta_{2n-1}^{-\epsilon}(3)\right| +12\left| \Lambda_{2n-1}^{-\epsilon}(3,-1+8\mathbb{Z})\right|+4\left| \Lambda_{2n-1}^{-\epsilon}(3,3+8\mathbb{Z})\right|.
\end{eqnarray*}

\noindent Let $x \in U(m)$, $(a_{1},\ldots,a_{2n-1}) \in \Lambda_{2n-1}^{-\epsilon}(3,x) \mapsto (-a_{1},\ldots,-a_{2n-1}) \in \Lambda_{2n-1}^{\epsilon}(3,-x)$ is a bijection (by proposition \ref{21} ii)). Thus, $\left| \Lambda_{2n-1}^{-\epsilon}(3,x)\right|=\left| \Lambda_{2n-1}^{\epsilon}(3,-x)\right|$. Besides, by corollary \ref{24}, for all $l \geq 5$ $\left|\Delta_{l}^{\epsilon}(3)\right|=\left|\Delta_{l}^{-\epsilon}(3)\right|$. Hence,
\begin{eqnarray*}
w_{2n,8} &=& 2\left|\Delta_{2n}^{\epsilon}(3)\right|+8\left|\Delta_{2n-1}^{\epsilon}(3)\right|+12\left| \Lambda_{2n-1}^{\epsilon}(3,-1+8\mathbb{Z})\right|+4\left| \Lambda_{2n-1}^{\epsilon}(3,3+8\mathbb{Z})\right| \\
				 & & +12\left| \Lambda_{2n-1}^{-\epsilon}(3,-1+8\mathbb{Z})\right|+4\left| \Lambda_{2n-1}^{\epsilon}(3,5+8\mathbb{Z})\right|. \\
				 &=& 2\left|\Delta_{2n}^{\epsilon}(3)\right|+8\left|\Delta_{2n-1}^{\epsilon}(3)\right|+12\left| \Lambda_{2n-1}^{\epsilon}(3,-1+8\mathbb{Z})\right|+4\left|\Delta_{2n-1}^{\epsilon}(3)\right| \\
				 & & +12\left| \Lambda_{2n-1}^{-\epsilon}(3,-1+8\mathbb{Z})\right|-4\left| \Lambda_{2n-1}^{\epsilon}(3,1+8\mathbb{Z})\right|-4\left| \Lambda_{2n-1}^{\epsilon}(3,-1+8\mathbb{Z})\right|. \\
				 &=& 2\left|\Delta_{2n}^{\epsilon}(3)\right|+12\left|\Delta_{2n-1}^{\epsilon}(3)\right|+12\left| \Lambda_{2n-1}^{\epsilon}(3,-1+8\mathbb{Z})\right|+12\left| \Lambda_{2n-1}^{-\epsilon}(3,-1+8\mathbb{Z})\right| \\
				 & & -4\left| \Lambda_{2n-1}^{-\epsilon}(3,-1+8\mathbb{Z})\right|-4\left| \Lambda_{2n-1}^{\epsilon}(3,-1+8\mathbb{Z})\right|. \\
				 &=& 2\left|\Delta_{2n}^{\epsilon}(3)\right|+12\left|\Delta_{2n-1}^{\epsilon}(3)\right|+8\left| \Lambda_{2n-1}^{\epsilon}(3,-1+8\mathbb{Z})\right|+8\left| \Lambda_{2n-1}^{-\epsilon}(3,-1+8\mathbb{Z})\right| \\
				 &=& 2\left|\Delta_{2n}^{\epsilon}(3)\right|+12\left|\Delta_{2n-1}^{\epsilon}(3)\right|+8w_{2n-2,8}~({\rm proposition}~\ref{25}). \\
\end{eqnarray*}

\noindent Hence, by corollary \ref{24}, we have the following equalities :
\begin{eqnarray*}
w_{2n,8} &=& 8^{n-2}w_{4,8}+\sum_{k=0}^{n-3} 8^{k}\frac{2^{4(n-k)-6}+7\times 2^{6(n-k)-9}}{3} \\
         &=& 28 \times 8^{n-2}+\frac{2^{4n-6}}{3}\sum_{k=0}^{n-3} \left(\frac{1}{2}\right)^{k} +\frac{7\times 2^{6n-9}}{3}\sum_{k=0}^{n-3} \left(\frac{1}{8}\right)^{k} \\
				 &=& 28 \times 8^{n-2}+\frac{2^{4n-5}}{3}(1-2^{2-n}) +\frac{2^{6n-6}}{3}(1-2^{6-3n}) \\
				 &=& 28 \times 8^{n-2}+\frac{2^{4n-5}-2^{3n-3}+2^{6n-6}-2^{3n}}{3}.
\end{eqnarray*}

\noindent Besides, this formula is already true for $n=2$.

\subsection{Numerical applications}
\label{chap24}

We can also have others formulas. Indeed, with the Chinese remainder theorem, we can easily prove the following results : 

\begin{corollary}
\label{28}

Let $k=p_{1}\ldots p_{r}$ with $p_{i}$ distinct prime numbers Let $m \geq 2$ and $n \geq 2$.  
\[w_{n,2^{m}k}^{+}:=|\{(a_{1},\ldots,a_{n})\in (\bZ/2^{m}k\bZ)^{n},~M_{n}(a_{1},\ldots,a_{n})=Id\}|=w_{n,2^{m}}^{+}u_{n,p_{1}}^{+}u_{n,p_{2}}^{+}\ldots u_{n,p_{r}}^{+}.\]
\[w_{n,2^{m}k}^{-}:=|\{(a_{1},\ldots,a_{n})\in (\bZ/2^{m}k\bZ)^{n},~M_{n}(a_{1},\ldots,a_{n})=-Id\}|=w_{n,2^{m}}^{-}u_{n,p_{1}}^{-}u_{n,p_{2}}^{-}\ldots u_{n,p_{r}}^{-}.\]

\end{corollary}

We now give some values obtained with the formulas given in the theorems \ref{prin1} and \ref{prin2}. We begin with $w_{n,N}^{+}$ for $n$ odd :

\hfill\break

\begin{center}
\begin{tabular}{|c|c|c|c|c|c|}
\hline
  \multicolumn{1}{|c|}{\backslashbox{$n$}{\vrule width 0pt height 1.25em$N$}}     & 8 & 16 & 24 & 32 & 40  \rule[-7pt]{0pt}{18pt} \\
  \hline
  3 & 1 & 1 & 1 & 1 & 1    \rule[-7pt]{0pt}{18pt} \\
	\hline
  5 & 80 & 320 & 800 & 1280 & 2080   \rule[-7pt]{0pt}{18pt} \\
	\hline
  7  & 5376 & 86~016 & 489~216 & 1~376~256 & 3~499~776   \rule[-7pt]{0pt}{18pt} \\
	\hline
  9  & 348~160 & 22~282~240 & 285~491~200 & 1~426~063~360 & 5~666~652~160  \rule[-7pt]{0pt}{18pt} \\
  \hline

\end{tabular}
\end{center}

\hfill\break

\noindent Now, we consider $w_{n,8}$ :

\hfill\break

\begin{center}
\begin{tabular}{|c|c|c|c|c|c|c|c|c|c|}
\hline
  $n$     & 2 & 3 & 4 & 5 & 6 & 7 & 8 & 9 & 10  \rule[-7pt]{0pt}{18pt} \\
  \hline
  $w_{n,8}$ & 1 & 2 & 28 & 160 & 1440 & 10~752 & 88~320 & 696~320 & 5~605~376   \rule[-7pt]{0pt}{18pt} \\
	\hline
	
\end{tabular}
\end{center}

\hfill\break

\noindent {\bf Acknowledgements}. I am grateful to Michael Cuntz for enlightening discussions.

\appendix

\end{document}